\documentclass[11pt,CJK,a4paper]{article}
\input amssymb.sty
\usepackage{amsmath}
\usepackage{amssymb}
\usepackage{bbm}
\usepackage{comment}  %comment 的作用是添加整段的评论  也可以隐藏不想显示的内容
\usepackage{extarrows}

\numberwithin{equation}{section}  %按章节编号

\evensidemargin=\oddsidemargin\textwidth=142mm \textwidth=160mm
\newtheorem{theorem}{Theorem}[section]
\newtheorem{lemma}{Lemma}[section]
\newtheorem{proposition}{Proposition}[section]

\newtheorem{definition}[theorem]{Definition}
\newtheorem{corollary}{Corollary}[section]

\newenvironment{proof}[1][Proof]{\begin{trivlist}
\item[\hskip \labelsep {\bfseries #1}]}{\end{trivlist}}

\def \D {\mathbbm{D}}
\def\p{\varphi}

\begin{document}

\begin{comment}
\title { 从Morrey空间到Bloch空间的积分算子\footnote{}\footnote{
基金项目; 福建省省属高校专项基金资助项目(JK2012010)；福建省自然科学基金资助项目(2009J0201)} }
\author{  卓钲渊\ , \  叶善力\thanks{通信作者; 叶善力(1967-~~),~ 男教授,~博士,~研究方向为复分析与调和分析.~shanliye@fjnu.edu.cn} \\
\small (福建师范大学数学与计算机科学学院\quad 福建福州 ~350007) \\
}
\date{}
\maketitle \noindent {\bf $\quad$摘要}\quad 本文讨论了从单位圆盘上的解析Morrey空间到Bloch空间的积分算子~$T_g$和$I_g$ 的有界性和本性范数. 而且给出了对$\mathcal{L}^{2,\lambda}$与$B$ 包含关系的猜测.
   \\
{\small\bf \quad关键词; }\quad  解析Morrey空间；Bloch空间；Volterra 型算子；本性范数\\
    {\small\bf MR（2010）主题分类:}  45P05, 42B35\\
   {\small\bf  中图分类号; }\quad O174.56 \quad\quad {\small\bf  文献标识码; }\quad A\\
\maketitle
\end{comment}

\title {Volterrra-type Operators  from  analytic Morrey spaces to  Bloch  space\footnote{ The article was  partially supported by  Special Fund of Colleges and Universities in Fujian Province (No: JK2012010) and  Natural Science Foundation of Fujian Province (No:2009J01004), China.}
}
\author{ Zhengyuan   Zhuo \quad\quad   Shanli Ye\footnote{Corresponding author.~ E-mail address:  shanliye@fjnu.edu.cn}\\
(\small \it Department of Mathematics, Fujian Normal University, Fuzhou 350007, P. R. China)
}

\date{}
\maketitle
\begin{abstract} In this note, we study the boundedness and compactness of integral operators $I_g$
and $T_g $ from analytic Morrey spaces to Bloch space. Furthermore, the norm and essential norm of those operators are given.
 \\
    {\small\bf 2010 MR Subject Classification }\quad 47B38, 30H30, 30H99 \\

\end{abstract}
\maketitle

\section{Introduction}
Let  $\D=\{z:|z|<1\}$ and $\partial\mathbbm{D}=\{z:|z|=1\}$  denote respectively the open unit disc and the unit
circle in the complex plane $\mathbbm{C}$. Let $H(\D)$ be the space of all analytic functions on
$\D$  and $dm(z)=\frac{1}{\pi}dxdy$ the normalized area Lebesgue measure.

The aim of this paper is to characterize the boundedness and compactness of two  Volterra type operators $I_g$ and $T_g$ from the analytic Morrey spaces $\mathcal{L}^{2,\lambda}$ to the classical Bloch space $B$, and from the little analytic Morrey spaces $\mathcal{L}^{2,\lambda}_0$ to the little Bloch space $B_0$.  Also, wee estimate the essential norm of $I_g$ and $T_g$.

  $Morrey$ space was initially introduced in 1938 by Morrey \cite{cbmj} to show that certain systems of partial differential equations (PDEs) had H\"{o}lder continuous solutions. In the past, $Morrey$ space has been studied heavily in different areas. For example, Adams and Xiao studied $Morrey$ spaces which is defined on Euclidean spaces $\mathbb{R}^n$ by potential theory and Hausdorff capacity in \cite{npaom,miha}. But here we will be mostly interested in  the analytic $Morrey$ spaces $\mathcal{L}^{2,\lambda}$ in the unit disk. It was introduced and studied by Wu and Xie in \cite{hwjz}.

  For an arc $I \subset \partial\mathbb{D}$, let $|I|=\frac{1}{2\pi}\int_I |d\zeta|$ be the normalized arc length of $I$, $$f_I =\frac{1}{|I|}\int_I f(\zeta)\frac{|d\zeta|}{2\pi}, f\in H(\D),$$
and $S(I)$ be the Carleson box based on $I$ with $$S(I)=\{z\in\D:1-|I|\leq|z|<1, \frac{z}{|z|}\in I\}.$$ Cleatly, if $I=\partial\mathbb{D}$, then $S(I)=\mathbb{D}$.

Denote  $\mathcal{L}^{2,\lambda}(\D)$ the analytic $Morrey$ spaces of all analytic functions $f\in H^2$ on $\mathbb{D}$ such that
$$\sup_{I \subset \partial\mathbbm{D}} \big(\frac{1}{|I|^\lambda} \int_I |f(\zeta)-f_I|^2\frac{|d\zeta|}{2\pi}
\big)^{1/2} <\infty,$$
where $0<\lambda\leq 1$ and the Hardy space $H^2$ consists of analytic functions f in $\mathbb{D}$ satisfying
$$
\sup_{0<r<1}\frac{1}{2\pi}\int_0^{2\pi}|f(re^{i\theta})|^2\,d\theta<\infty.
$$

Similarly to the relation between $BMOA$ space and $VMOA$ space, we have that $f\in \mathcal{L}^{2,\lambda}_0(\D)$, the little $Morrey$ spaces, if $f\in \mathcal{L}^{2,\lambda}(\D)$ and
$$\lim_{|I| \rightarrow 0} \big(\frac{1}{|I|^\lambda} \int_I |f(\zeta)-f_I|^2\frac{|d\zeta|}{2\pi}
\big)^{1/2} =0.$$

Xiao and Xu \cite{cobacs} studied the composition operators of $\mathcal{L}^{2,\lambda}$ spaces. Cascante, F\`{a}brega and Ortega \cite{tctiwh} studied the Corona theorem of $\mathcal{L}^{2,\lambda}$. It is a useful tools for the study of harmonic analysis and partial differential equations, We refer the readers to \cite{cbmj,jp,ctz}.

The following lemma gives some equivalent conditions of $\mathcal{L}^{2,\lambda}(\D)$ (see Theorem 3.21 of \cite{xia1} or Theorem 3.1 of \cite{wul}).
\begin{lemma}
 Suppose that $0<\lambda<1$ and $f\in H(\D)$. Let $a\in \D$, $\p_a (z)=\frac{a-z}{1-\overline{a}z}$. Then the
following statements are equivalent.

$(i)~f\in \mathcal{L}^{2,\lambda}(\D)$;

$(ii)~\displaystyle\sup_{I \subset \partial\mathbbm{D}} \frac{1}{|I|^\lambda} \int_{S(I)} |f'(z)|^2(1-|z|^2)dm(z) <\infty$;

$(iii)~\displaystyle\sup_{a\in \D} (1-|a|^2)^{1-\lambda}\int_{\D} |f'(z)|^2(1-|\p_a (z)|^2)dm(z) <\infty$.

\end{lemma}
 From the lemma above, we can define the norm of function
$f\in \mathcal{L}^{2,\lambda}(\D)$ and equivalent formula as follows
\begin{eqnarray*}
\|f\|_{\mathcal{L}^{2,\lambda}}&=&|f(0)|+\sup_{I \subset \partial\mathbbm{D}} \big(\frac{1}{|I|^\lambda} \int_{S(I)} |f'(z)|^2(1-|z|^2)dm(z)\big)^{1/2}\\
&\approx& |f(0)|+\sup_{a\in \D} \big((1-|a|^2)^{1-\lambda}\int_{\D} |f'(z)|^2(1-|\p_a (z)|^2)dm(z)\big)^{1/2}.
\end{eqnarray*}
 It is  known that $\mathcal{L}^{2,1}(\D)=BMOA$ and if $0<\lambda<1$, $BMOA\subsetneq\mathcal{L}^{2,\lambda}(\D)$. For more information on $BMOA$ and $VMOA$, see \cite{gir}.

A function ~$f$ analytic on the unit disk is said to belong to the ~$Bloch$ space ~$B$ if
$$
     \|f\|_B =\sup_{z\in\mathbb{D} } \{(1-|z|^2)|f'(z)|  \}< \infty,
$$
and to the little ~$Bloch$ space ~$B_0$ if  $f\in B$ and
$$\lim_{|z|\to 1^-}(1-|z|^2)|f'(z)|= 0.$$

It is well known  that $B$ is a Banach space under the norm
$\||f\||_B=|f(0)| + \|f\|_B$ and ~$ B_0$ is a closed subspace of ~$B$. See ${\cite{bf}}$. By \cite{arlp,jx}, together with Lemma 2.1 in \cite{radsjx}, we have the following equivalent statements about the norm of $f\in B $.
\begin{proposition}For all $p\in(1,\infty)$,
\begin{eqnarray*}
\label{mt2}
\||f\||_B  &\approx& |f(0)|+\sup_{a\in \D} \big(\int_{\D} |f'(z)|^2(1-|\p_a (z)|^2)^p dm(z)\big)^{1/2}\\
&\approx&|f(0)|+\sup_{I \subset \partial\mathbbm{D}} \big(\frac{1}{|I|^p} \int_{S(I)} |f'(z)|^2(1-|z|^2)^{p}dm(z)\big)^{1/2}.
\end{eqnarray*}
\end{proposition}

Suppose that $g:\D\to \mathbb{C}$  is a holomorphic map. The integral operator $T_g$, called  Volterra-type operator, is defined as $$T_g f(z)=\int_0^z f(w)g'(w)dw, ~~~z\in\D, ~~f\in H(\D).$$

In \cite{chp} Pommerenke introduced the operator $T_g$ and showed that $T_g$ is a bounded operator on the Hardy space $H^2$ if and only if $g \in BMOA$.

The companion integral operator $I_g$ is spontaneously defined as $$I_g f(z)=\int_0^z f'(w)g(w)dw,~~~z\in\D, ~~f\in H(\D).$$

The boundedness, compactness or essential norm of $T_g$ and $I_g$ between spaces of analytic functions were investigated by many authors.
Aleman and Siskakis in \cite{aaags} studied the integral operator $T_g$ on the Bergman space, and then Aleman considered with Cima $T_g$ acting on the Hardy space in \cite{aajac}. Siskakis and Zhao \cite{agsrz} also investigated $T_g$ on the space $BMOA$. $T_g$ on the $Q_p$ space was studied by Xiao in \cite{jxqcm} . Li and Stevi\'{c} in \cite{slss} studied the boundedness and compactness of $T_g$ and $I_g$ on the $Zygmund$ Spaces and the little $Zygmund$ spaces. Cinstantin in \cite{oc} considered the boundedness and compactness of $T_g$ on $Fock$ spaces. Ye in \cite{povo} studied products of Volterra-type operators and composition operators on logarithmic $Bloch$ space. Ye and Gao in \cite{ysgj} gave the boundedness and compactness of $T_g$ between different weighted $Bloch$ spaces. 

 There are some articles about the integral operator acting on $Morrey$ space. For example, Wu in \cite{zw} considered $T_g$ from Hardy to analytic $Morrey$ spaces, Li, Liu and  Lou \cite{pljlzl} characterized the boundedness and essential norm of $T_g$ and $I_g$ on analytic $Morrey$ spaces (see also the related references therein).

Now, we need two spaces. Let $\alpha >-1$. Recall that $f\in H(\D)$ belongs to the weighted  space $H_\alpha^\infty$ if it satisfies with
$$\sup_{z\in \D} (1-|z|^2)^\alpha |f(z)|<\infty.$$
When $\alpha >-1$, $H_\alpha^\infty$ endowed with the norm $\|f\|_\alpha=\sup_{z\in \D} (1-|z|^2)^\alpha |f(z)|$ is a Banach space. This space is connected with the study of growth conditions of analytic functions and was also studied in detail , see \cite{bbt,kabjbag,sw1,sw2}.
The space $H_\alpha^\infty$ is used in the  characterizations of the boundedness and essential norm of $I_g$. Then we conclude the boundedness and essential norm of $T_g$ by introducing the following $Bloch-Morrey$ type space.

\begin{definition}
Let $0<\lambda\leq1$ and $p> 1$. The $Bloch-Morrey$ type space $B\mathcal{L}^{p,\lambda}$ is the set of all $g\in H(\D)$ such that
$$M(g)=\sup_{I \subset \partial\mathbbm{D}} \big(\frac{1}{|I|^{p-\lambda+1}} \int_{S(I)} |g'(z)|^2(1-|z|^2)^{p}dm(z)\big)^{1/2}<\infty.$$
The corresponding subspace $B\mathcal{L}^{p,\lambda}_0$, the little $Bloch-Morrey$ type space, can be defined as
$$B\mathcal{L}^{p,\lambda}_0=\{g\in B\mathcal{L}^{p,\lambda},~~\lim_{|I| \rightarrow 0} \big(\frac{1}{|I|^{p-\lambda+1}} \int_{S(I)} |g'(z)|^2(1-|z|^2)^{p}dm(z)\big)^{1/2}=0\}.$$
\end{definition}

It is easy to prove that $B\mathcal{L}^{p,\lambda}$ is a Banach space under the norm$$\|g\|_{B\mathcal{L}^{p,\lambda}}=|g(0)|+M(g).$$
Clearly, $B\mathcal{L}^{p,1}=Bloch$. From \cite{radsjx}, we know that $\|g\|_{B\mathcal{L}^{p,\lambda}}$ is comparable with the norm
$$ |g(0)|+\sup_{a\in \mathbbm{D}} \big( \int_{\D}(\frac{1-|a|^2}{|1-\overline{a}z|^2})^{p+1-\lambda} |g'(z)|^2(1-|z|^2)^{p}dm(z)\big)^{1/2}.$$

$Notations$: For two functions $F$ and $G$, if there is a constant $C > 0$ dependent only on indexes $p,\lambda...$  such that $F\leq CG$, then we say that $F\lesssim G$. Furthermore, denote that $F\thickapprox G$ ($F$ is comparable with $G$) whenever $F\lesssim G \lesssim F$.

\section{$\mathcal{L}^{2,\lambda}$ vs $B$}%$\mathcal{L}^{2,\lambda}$与$B$关系猜测

Evidently, when $0<\lambda<1$, $BMOA\subsetneq\mathcal{L}^{2,\lambda}(\D)$. On the other hand $BMOA\subsetneq B$. Does $\mathcal{L}^{2,\lambda}$ and $B$ have the inclusion relation ? We claim the answer is negative by the following two proposition. This makes our job more signality.
\begin{proposition}
$\mathcal{L}^{2,\lambda}\nsubseteq B.$
\end{proposition}

\begin{proof}
Considering $g(z)=(\log \frac{1}{1-z})^2$, which is obviously not a $Bloch$ function. We claim that $g(z)\in \mathcal{L}^{2,\lambda}$. Indeed,
\begin{eqnarray*}
\lefteqn{  \sup_{a\in \D} (1-|a|^2)^{1-\lambda}\int_{\D} |g'(z)|^2(1-|\rho_a (z)|^2)dm(z)} \hspace{30pt}\\
&\lesssim &\sup_{a\in \D} (1-|a|^2)^{1-\lambda}\int_{\D} |\frac{1}{1-z}|^2\log^2\frac{1}{1-|z|}(1-|\rho_a (z)|^2)dm(z)\\
&\lesssim & \int_{\D} |\frac{1}{1-z}|^2 (1-|z|^2)^{1-\lambda} \log^2\frac{1}{1-|z|}dm(z)\\
&= & \int_{0}^1 \int_{0}^{2\pi} |\frac{1}{1-re^{i\theta}}|^2 d\theta (1-r^2)^{1-\lambda}\log^2\frac{1}{1-r}dr\\
&= & \int_{0}^1  (1-r^2)^{-\lambda}\log^2\frac{1}{1-r}dr < \infty.
\end{eqnarray*}

This finish the proof.
\end{proof}

Conversely, the function $\displaystyle f(z)=\sum_{n=0}^\infty z^{2^n}$ is a $Bloch$ function (see \cite{obf}) and it is well known that it has a radial limit almost nowhere. Consequently, $f$ does not belong to any of the $Hardy$ spaces and so $f \notin \mathcal{L}^{2,\lambda}$. So we have the following proposition.
\begin{proposition}
$B\nsubseteq \mathcal{L}^{2,\lambda}.$
\end{proposition}

\section{Boundedness of $I_g$ and $T_g$ from $\mathcal{L}^{2,\lambda}$ to $B$}
In this section, we prove the boundedness and estimate the norms of $I_g$ and $T_g$. The following lemmas will be used through this paper.
\begin{lemma}
\label{314}
Let $0<\lambda<1$ and $b\in\D$. We set functions $f_b (z)$ and  $F_b (z)$ as
$$f_b (z)=(1-|b|^2)^{\frac{1-\lambda}{2}}(\rho_b (z)-b),~~~F_b (z)=(1-|b|^2)(1-\overline{b}z)^{\frac{\lambda-3}{2}}.$$
then $f_b (z)\in \mathcal{L}^{2,\lambda}(\D)$ and $F_b (z)\in \mathcal{L}^{2,\lambda}(\D)$. Particularly, we have $f_b (z)\in \mathcal{L}^{2,\lambda}_0(\D)$ and $F_b (z)\in \mathcal{L}^{2,\lambda}_0(\D)$. Moreover, $\|f_b\|_{\mathcal{L}^{2,\lambda}}\lesssim 1$,
$\|F_b\|_{\mathcal{L}^{2,\lambda}} \lesssim 1$.
\end{lemma}
\begin{proof} See Lemma 4 in \cite{pljlzl}. From its proof, we futher deduce that
$f_b (z)\in \mathcal{L}^{2,\lambda}_0(\D)$ and $F_b (z)\in \mathcal{L}^{2,\lambda}_0(\D)$.
\end{proof}

we get a result about the growth rate of functions in
$\mathcal{L}^{2,\lambda}(\D)$ from \cite{pljlzl}.
\begin{lemma}
\label{mbd}
Let $0<\lambda<1$. If $f\in \mathcal{L}^{2,\lambda}(\D)$. then
$$
|f(z)|\lesssim \frac{\|f\|_{\mathcal{L}^{2,\lambda}}} {(1-|z|^2)^\frac{1-\lambda}{2}},~~~~~z\in \D.
$$
\end{lemma}

We first consider the boundedness of $I_g:\mathcal{L}^{2,\lambda} \rightarrow B$ .

\begin{theorem}
\label{3dl1}
Let $0<\lambda<1$ and $g\in H(\D)$. Then $I_g:\mathcal{L}^{2,\lambda} \rightarrow B$ is bounded if and only if
$g\in H_{ \frac{\lambda-1}{2} }^\infty.$ Moreover the operator norm satisfies $$
\|I_g\| \thickapprox \|g\|_\frac{\lambda-1}{2}.$$
\end{theorem}

\begin{proof}
For $0<\lambda<1$, $1<2-\lambda$, we set $B=Q_{2-\lambda}.$

Sufficiency: let $g\in H_{ \frac{\lambda-1}{2} }^\infty$. For any $f\in \mathcal{L}^{2,\lambda}(\D)$, we have
\begin{eqnarray*}
\|I_g f\|_B
&\approx& \sup_{a\in \D} \big(\int_{\D} |f'(z)|^2|g(z)|^2(1-|\varphi_a (z)|^2)^{2-\lambda} dm(z)\big)^{1/2}\\
&=&\sup_{a\in \D} \big( (1-|a|^2)^{1-\lambda} \int_{\D} |f'(z)|^2(1-|\varphi_a (z)|^2)|g(z)|^2 (\frac{1-|z|^2}{|1-\overline{a}z|^2})^{1-\lambda}dm(z)\big)^{1/2}\\
&\lesssim&\|g\|_\frac{\lambda-1}{2} \cdot \|f\|_{\mathcal{L}^{2,\lambda}}.
\end{eqnarray*}
These inequalities imply $I_g$ is bounded and $\|I_g\| \lesssim \|g\|_\frac{\lambda-1}{2}.$

Necessity: let $I_g$ is bounded. For any $b\in\D$, considering functions $f_b (z)$ in Lemma \ref{314}, we have $\|f_b\|_{\mathcal{L}^{2,\lambda}}\lesssim 1$. Thus
\begin{eqnarray}
\label{31}
\|I_g \| &\gtrsim& \|I_g f_b\|_B  \nonumber \\
&\approx& \sup_{a\in \D} \big(\int_{\D} |f_b'(z)|^2|g(z)|^2(1-|\varphi_a (z)|^2)^{2-\lambda} dm(z)\big)^{1/2}   \nonumber\\
&=& \sup_{a\in \D} \big(\int_{\D} \frac{(1-|b|^2)^{1+\lambda}}{|1-\overline{b}z|^4}|g(z)|^2(1-|\varphi_a (z)|^2)^{1+1-\lambda} dm(z)\big)^{1/2} \nonumber   \\
&\geqslant&   \big( \int_{\D} \frac{(1-|b|^2)^2}{|1-\overline{b}z|^4}|g(z)|^2 (\frac{1-|z|^2}{|1-\overline{b}z|^2})^{1-\lambda} (1-|\rho_b (z)|^2) dm(z)\big)^{1/2}\\
&= & \big( \int_{\D} |\rho_b '(z)|^2 |g(z)|^2 (\frac{1-|z|^2}{|1-\overline{b}z|^2})^{1-\lambda} (1-|\rho_b (z)|^2) dm(z)\big)^{1/2}  \nonumber \\
&= &\big( \int_{\D}  |g(\rho_b(w))|^2 (\frac{1-|\rho_b(w)|^2}{|1-\overline{b}\rho_b(w)|^2})^{1-\lambda} (1-|w|^2) dm(w)\big)^{1/2} \nonumber \\
&\gtrsim &|\frac{g(b)}{(1-|b|^2)^{\frac{1-\lambda}{2}}}|. \nonumber
\end{eqnarray}
where we used Lemma 4.12 of $\cite{otifs}$  in the last inequality. Since $b$ is arbitrary, we have $\|I_g \|\gtrsim \|g\|_\frac{\lambda-1}{2}.$  The proof is finished.
\end{proof}

With the space $B\mathcal{L}^{p,\lambda}$, we can establish the boundedness of $T_g:\mathcal{L}^{2,\lambda} \rightarrow B$ as the following theorem.

\begin{theorem}
\label{dl2}
Suppose that $0<\lambda<1$ and $g\in H(\D)$. Then the following conditions are equivalent:

$(i)$ $T_g:\mathcal{L}^{2,\lambda} \rightarrow B$ is bounded;

$(ii)$ $g\in B\mathcal{L}^{p,\lambda}$ for all $p\in (1,\infty)$;

$(iii)$ $g\in B\mathcal{L}^{p,\lambda}$ for some $p\in (1,\infty)$ .

Moreover, $$\|T_g\| \thickapprox M(g).$$
\end{theorem}

\begin{proof}

$(i)\Rightarrow (ii)$.  Suppose that $T_g:\mathcal{L}^{2,\lambda} \rightarrow B$  is bounded. For any $I \subset \partial\mathbbm{D}$, let $b=(1-|I|)\zeta  \in \D$, where $\zeta$ is the centre of $I$. Then
$$(1-|b|^2)\approx |1-\overline{b}z| \approx |I|, ~~z\in S(I).$$
Considering the functions $F_b (z)$ in Lemma \ref{314}, $\|F_b\|_{\mathcal{L}^{2,\lambda}} \lesssim 1$. This together with Proposition \ref{mt2}, we obtain that for any $p\in (1,\infty)$,
\begin{eqnarray*}
\frac{1}{|I|^{p-\lambda+1}} \int_{S(I)} |g'(z)|^2(1-|z|^2)^{p}dm(z)
&\thickapprox& \frac{1}{|I|^p} \int_{S(I)} |F_b(z)|^2|g'(z)|^2(1-|z|^2)^{p}dm(z)\\
&\lesssim& \|T_g F_b\|_B^2\\
&\leq& \|T_g\|^2\|F_b\|^2_{\mathcal{L}^{2,\lambda}}\\
&\lesssim&\|T_g\|^2.
\end{eqnarray*}
Since $I$ is arbitrary, we have $M(g)\lesssim\|T_g\|$.

$(ii)\Rightarrow (iii)$. It is obvious.

$(iii)\Rightarrow (i)$.  Suppose that fixed $p\in (1,\infty)$ and $M(g)<\infty$. For $f\in \mathcal{L}^{2,\lambda}(\D)$ and any $I \subset \partial\mathbbm{D}$, from Lemma $\ref{mbd}$, it follows that
\begin{eqnarray}
\begin{split}
\label{32}
\|T_g f\|_B
&\approx& \sup_{I \subset \partial\mathbbm{D}} \big(\frac{1}{|I|^p} \int_{S(I)}|f(z)|^2 |g'(z)|^2(1-|z|^2)^{p}dm(z)\big)^{1/2}~~~~~~~~~~\\
&\lesssim& \|f\|_{\mathcal{L}^{2,\lambda}}  \cdot \sup_{I \subset \partial\mathbbm{D}} \big(\frac{1}{|I|^p} \int_{S(I)} |g'(z)|^2(1-|z|^2)^{p+\lambda-1}dm(z)\big)^{1/2}.
\end{split}
\end{eqnarray}

 To the end, for a given subarc $I$ of $\partial\mathbbm{D}$, let $\mathcal{D}_n(I)$  represent the set of $2^n$  subarcs of length
$2^{-n}|I|$ obtained by $n$ successive bipartition of $I$. For each $J\in \mathcal{D}_n(I)$ write $T(J)$ for the top half Carleson box of $S(J)$, i.e.,$$T(J)=\{z\in S(J):\frac{z}{|z|}\in J, 1-|J|<|z|<1-\frac{|J|}{2}\}.$$

Then
$$S(I)=\mathop{\bigcup}\limits_{n=0}^\infty \mathop{\bigcup}\limits_{J\in \mathcal{D}_n(I)}T(J). $$
Noting that $z\in T(J)$, $1-|z|\approx |J|$,
one has
\begin{eqnarray}
\label{33}
\int_{S(I)} |g'(z)|^2(1-|z|^2)^{p+\lambda-1}dm(z)
&=&\sum_{n=0}^\infty\sum_{J\in \mathcal{D}_n(I)}\int_{T(J)} |g'(z)|^2(1-|z|^2)^{p+\lambda-1}dm(z)    \nonumber \\
&\approx&\sum_{n=0}^\infty\sum_{J\in \mathcal{D}_n(I)}\int_{T(J)} |J|^{\lambda-1}|g'(z)|^2(1-|z|^2)^{p}dm(z)  \nonumber  \\
&\leq&\sum_{n=0}^\infty\sum_{J\in \mathcal{D}_n(I)}\int_{S(J)} |J|^{\lambda-1}|g'(z)|^2(1-|z|^2)^{p}dm(z)  \nonumber\\
&\leq& \sum_{n=0}^\infty\sum_{J\in \mathcal{D}_n(I)}M(g)^2 |J|^{\lambda-1}|J|^{p-\lambda+1}\\
&=& \sum_{n=0}^\infty 2^n M(g)^2|J|^p   \nonumber   \\
&=& \sum_{n=0}^\infty (2^n)^{1-p}M(g)^2|I|^p  \nonumber\\
&\lesssim& M(g)^2|I|^p.\nonumber
\end{eqnarray}
Now invoking $(\ref{32})$,
\begin{eqnarray*}
\|T_g f\|_B\lesssim M(g)\cdot\|f\|_{\mathcal{L}^{2,\lambda}}
\end{eqnarray*}
As a result, $\|T_g\|\lesssim M(g)$.

\end{proof}

Theorem $\ref{dl2}$ has an interesting consequence.
\begin{corollary}
\label{tl1}
Let $0<\lambda<1$ and $1<p<q<\infty$. Then $B\mathcal{L}^{p,\lambda}=B\mathcal{L}^{q,\lambda}$.
\end{corollary}

\section{Essential norm of $I_g$ and $T_g$ from $\mathcal{L}^{2,\lambda}$ to $B$ }
Let $X$ and $Y$ be Banach spaces. The essential norm  of a bounded operator $T:X \rightarrow Y$ ,$\|T\|_{e,X \rightarrow Y} $, is
defined as the distance from $T$ to the space of compact operators, $$\|T\|_{e,X \rightarrow Y}=inf\{\|T-K\|_{X \rightarrow Y}:~~~K~is~any~compact ~operator \},$$where the norm of $T$ is denoted by $\|\cdot\|_{X \rightarrow Y} $.

Since that $T$ is compact if and only if $\|T\|_{e,X \rightarrow Y} =0$, then the estimation of $\|T\|_{e,X \rightarrow Y} $
indicates the condition for $T$ to be compact.  For some recent results related to  the essential norm , see
\cite{jlsmpn,jlzlcx,js,pljlzl}. and the references therein.

In this section, we estimate {the essential norm of $I_g$ and $T_g$ from $\mathcal{L}^{2,\lambda}$ to $B$.  We need some auxiliary results.

\begin{lemma}
\label{yl5}
Let $0<\lambda<1$. For $0<t<1$, $z\in \D$, $f_t(z)=f(tz)$. If $f\in \mathcal{L}^{2,\lambda}(\D)$, then
$f_t\in \mathcal{L}^{2,\lambda}_0 (\D)$ and $\|f_t\|_{\mathcal{L}^{2,\lambda}}\leq \|f\|_{\mathcal{L}^{2,\lambda}}$.
\end{lemma}
\begin{proof}
If $f\in \mathcal{L}^{2,\lambda}(\D)$ and $0<t<1$, then, $f_t$ is analytic on the closed unit disk $\bar{\D}$. A simple computation shows that $f_t\in \mathcal{L}^{2,\lambda}_0 (\D)$.
 In addition, by Poisson formula, we have $$
f_t (z)=\int_0^{2\pi} f(ze^{i\theta})\frac{1-t^2}{|e^{i\theta}-t|^2}\frac{d\theta}{2\pi},~~~z\in\D.
$$
So,\begin{eqnarray*}
\lefteqn{ \sup_{a\in \D} (1-|a|^2)^{1-\lambda}\int_{\D} |f_t '(z)|^2(1-|\varphi_a (z)|^2)dm(z)   } \hspace*{20pt}\\
&\leq&     \sup_{a\in \D} (1-|a|^2)^{1-\lambda}\int_{\D} \int_0^{2\pi}|f'(ze^{i\theta})|^2\frac{1-t^2}{|e^{i\theta}-t|^2}\frac{d\theta}{2\pi}(1-|\varphi_a (z)|^2)dm(z) \\
&=&    \int_0^{2\pi}\sup_{a\in \D} (1-|a|^2)^{1-\lambda} \int_{\D} |f'(ze^{i\theta})|^2
(1-|\varphi_a (z)|^2)dm(z) \frac{1-t^2}{|e^{i\theta}-t|^2}\frac{d\theta}{2\pi}\\
&\leq& \sup_{a\in \D} (1-|a|^2)^{1-\lambda} \int_{\D} |f'(z)|^2
(1-|\varphi_a (z)|^2)dm(z) \cdot \int_0^{2\pi} \frac{1-t^2}{|e^{i\theta}-t|^2}\frac{d\theta}{2\pi}\\
&=&\sup_{a\in \D} (1-|a|^2)^{1-\lambda} \int_{\D} |f'(z)|^2
(1-|\varphi_a (z)|^2)dm(z).
\end{eqnarray*}
Thus, $\|f_t\|_{\mathcal{L}^{2,\lambda}}\leq \|f\|_{\mathcal{L}^{2,\lambda}}$.
\end{proof}

By Lemma $\ref{mbd}$  and standard arguments(see, e.g., \cite{coosoaf},Proposition 3.11), the following lemma follows.
\begin{lemma}
Assume that g is an analytic function on $\D$. Then  $T_g(or I_g):\mathcal{L}^{2,\lambda} \rightarrow B$ is compact if and only if $T_g(or I_g):\mathcal{L}^{2,\lambda} \rightarrow B$ is bounded, and for any bounded sequence $(f_k)_{k\in \mathbb{N}}$ in $\mathcal{L}^{2,\lambda}$ which converges to zero uniformly on $\D$ as $k\to \infty$, $\|T_g f_k\|_B \to 0 (or~\|I_g f_k\|_B \to 0)$ ~as $k\to \infty$.
\end{lemma}

\begin{lemma}
\label{yl6}
Suppose that $0<\lambda<1$ and $p>1$. For $g\in B\mathcal{L}^{p,\lambda}$, define the following  operators $T_{g,r}:\mathcal{L}^{2,\lambda} \rightarrow B$ :$$
T_{g,r} f(z)=\int_0^z f(rw)g'(w)dw.
$$where $r\in (0,1)$. Then $T_{g,r}$ is compact.
\end{lemma}
\begin{proof}
Let $\{f_n\}$ be such that $\|f_n\|_{\mathcal{L}^{2,\lambda}}\leq 1 $ and $f_n\to 0$ uniformly on compact subsets of $\D$ as $n \to \infty$. We are required to show that $\displaystyle \lim_{n\rightarrow \infty} \|T_{g,r}f_n\|_B=0.$
In fact, Since $\|g\|_B\lesssim M(g)$, we have $|g'(z)|\lesssim \frac{M(g)}{1-|z|^2}$. From $\|f_n\|_{\mathcal{L}^{2,\lambda}}\leq 1 $ and Lemma $\ref{mbd}$ , it yields that $(1-|r|^2)^{(1-\lambda)/2} |f_n (rz)|\lesssim 1$.
Thus \begin{eqnarray*}
\|T_{g,r}f_n\|_B&=&\sup_{z\in \D} (1-|z|^2) |f_n (r z)||g'(z)|\\
&\lesssim& M(g) \sup_{z\in \D}  \frac{1}{(1-r^2)^{(1-\lambda)/2}}.
\end{eqnarray*}
Accordingly, by Dominated Convergence Theorem one reaches $\displaystyle \lim_{n\rightarrow \infty} \|T_{g_r}f_n\|_B=0.$

\end{proof}

Now, we present the  main result of this section.

\begin{theorem}
\label{dl3}
Suppose $ 0<\lambda<1$ and $g\in H(\D)$. If $I_g:\mathcal{L}^{2,\lambda} \rightarrow B$ is bounded, then$$
\|I_g\|_{e, \mathcal{L}^{2,\lambda}\to B} \thickapprox \|g\|_\frac{\lambda-1}{2}.$$
\end{theorem}
\begin{proof}
Choose the zero operator $O:\mathcal{L}^{2,\lambda} \rightarrow B:f\mapsto 0$. Since $O$ is compact and $\|O\|=0$, we  get $$
\|I_g\|_{e, \mathcal{L}^{2,\lambda}\to B} =\inf_K \|I_g-K\|\leq\|I_g\|\lesssim \|g\|_\frac{\lambda-1}{2}.$$

Conversely, choose the sequence $\{b_n\}\subset \D$ such that $|b_n|\rightarrow 1$ as $n\rightarrow \infty$. Considering the sequence of functions $f_n (z)=(1-|b_n|^2)^{\frac{1-\lambda}{2}}(\rho_{b_n} (z)-b_n)$, we obtain $\|f_n\|_{\mathcal{L}^{2,\lambda}}\lesssim 1$ by Lemma \ref{314}.
By easy calculation, $f_n (z)=-(1-|b_n|^2)^{\frac{1+\lambda}{2}} \int_{0}^z \frac{dw}{(1-\overline{b_n}z)^2}$, and thus $f_n$ converges to zero uniformly on compact subsets of $\D$. Then $\|Kf_n\|_B\rightarrow 0$ as $n\rightarrow \infty$  for any compact operator $K$. So
$$
\|I_g-K\|\gtrsim \limsup_{n\rightarrow\infty} \|(I_g-K)f_n\|_B
\geq\limsup_{n\rightarrow\infty} (\|I_g f_n\|_B-\|Kf_n\|_B)
\geq \limsup_{n\rightarrow\infty} \|I_g f_n\|_B.
$$
By $(\ref{31})$, we have $$\|I_g-K\|\gtrsim \limsup_{n\rightarrow\infty} |\frac{g(b_n)}{(1-|b_n|^2)^{\frac{1-\lambda}{2}}}|.$$
The arbitrary choice of the sequence $\{b_n\}$ implies $$
\|I_g\|_{e, \mathcal{L}^{2,\lambda}\to B} \gtrsim \|g\|_\frac{\lambda-1}{2}.$$
\end{proof}

\begin{theorem}
\label{dl4}
 Suppose $0<\lambda<1$ and $g\in H(\D)$. If $T_g:\mathcal{L}^{2,\lambda} \rightarrow B$ is bounded, then $$
\|T_g\|_{e, \mathcal{L}^{2,\lambda}\to B} \thickapprox \limsup_{|a|\rightarrow 1} \big( \int_{\D}(\frac{1-|a|^2}{|1-\overline{a}z|^2})^{p+1-\lambda} |g'(z)|^2(1-|z|^2)^{p}dm(z)\big)^{1/2}.$$
\end{theorem}
\begin{proof}
For any $r_n\in(0,1)$  such that $r_n\rightarrow 1$ as $n\rightarrow \infty$, we introduce $T_{g,r_n}:\mathcal{L}^{2,\lambda} \rightarrow B$ which is compact. Let $s\in (0,1)$ , we have
\begin{eqnarray*}
\lefteqn {  \|T_g\|_{e, \mathcal{L}^{2,\lambda}\to B}}\hspace{30pt} \\
&\leq& \|T_g-T_{g,r_n}\| \\
&\approx&\sup_{\|f\|_{\mathcal{L}^{2,\lambda}}= 1}  \|T_g-T_{g,r_n}\|_B \\
&=&\sup_{\|f\|_{\mathcal{L}^{2,\lambda}}= 1}   \sup_{a\in \D} \big(\int_{\D} |f(z)-f(r_n z)|^2|g'(z)|^2(1-|\varphi_a (z)|^2)^{p+1-\lambda } dm(z)\big)^{1/2}\\
&\leq&  \sup_{\|f\|_{\mathcal{L}^{2,\lambda}}= 1}   \sup_{|a|\leq s} \big(\int_{\D} |f(z)-f(r_n z)|^2|g'(z)|^2(1-|\varphi_a (z)|^2)^{p+1-\lambda } dm(z)\big)^{1/2}\\
&~~+&\sup_{\|f\|_{\mathcal{L}^{2,\lambda}}= 1}   \sup_{|a|>s} \big(\int_{\D} |f(z)-f(r_n z)|^2|g'(z)|^2(1-|\varphi_a (z)|^2)^{p+1-\lambda } dm(z)\big)^{1/2}\\
&\leq&  \sup_{\|f\|_{\mathcal{L}^{2,\lambda}}= 1}   \sup_{|a|\leq s} \big(\int_{\D} |f(z)-f(r_n z)|^2|g'(z)|^2(1-|\varphi_a (z)|^2)^{p+1-\lambda } dm(z)\big)^{1/2}\\
&~~+& 2  \sup_{|a|>s} \int_{\D}(\frac{1-|a|^2}{|1-\overline{a}z|^2})^{p+1-\lambda} |g'(z)|^2(1-|z|^2)^{p}dm(z)\big)^{1/2}\\
&\triangleq & K_1+K_2.
\end{eqnarray*}
Since $|a|\leq s$ is a closed set of $\D$ and $g\in B\mathcal{L}^{p,\lambda}$, the Dominated Convergence Therorem yields $K_1\rightarrow 0$ as $n\rightarrow \infty$.

Now, letting $n\rightarrow \infty$ and then letting $s\rightarrow 1$, we get
$$\|T_g\|_{e, \mathcal{L}^{2,\lambda}\to B}\lesssim  \limsup_{|a|\rightarrow 1} \big( \int_{\D}(\frac{1-|a|^2}{|1-\overline{a}z|^2})^{p+1-\lambda} |g'(z)|^2(1-|z|^2)^{p}dm(z)\big)^{1/2}.$$

Conversely, Let $I_n \subset \partial\mathbbm{D}$ such that  $|I_n|\rightarrow 0$ as $n\rightarrow \infty$. $\zeta_n$ is the center of arc $I$  and $b_n=(1-|I_n|)\zeta$, so
$$(1-|b_n|^2)\approx |1-\overline{b_n}z| \approx |I_n|, ~~z\in S(I_n).$$
Consider the function $F_n (z)=(1-|b_n|^2)(1-\overline{b_n}z)^{\frac{\lambda-3}{2}}$. Then $\|F_n\|_{\mathcal{L}^{2,\lambda}} \lesssim 1$ and $F_n\to 0$ uniformly on the compact subsets of $\D$ as $n\rightarrow \infty$ by Lemma $\ref{314}$ . Thus $\|KF_n\|_B\rightarrow 0$ for any compact operator $K$.
Therefore
\begin{eqnarray*}
\|T_g-K\|\gtrsim \limsup_{n\rightarrow\infty} \|(T_g-K)F_n\|_B
\geq\limsup_{n\rightarrow\infty} (\|T_g F_n\|_B-\|KF_n\|_B)
\geq \limsup_{n\rightarrow\infty} \|T_g f_n\|_B\\
\gtrsim \limsup_{n\rightarrow\infty} \big(\frac{1}{|I_n|^p} \int_{S(I_n)} |F_n(z)|^2|g'(z)|^2(1-|z|^2)^{p}dm(z)\big)^{1/2}\\
\thickapprox \limsup_{n\rightarrow\infty} \big(\frac{1}{|I_n|^{p+1-\lambda}} \int_{S(I_n)} |g'(z)|^2(1-|z|^2)^{p}dm(z)\big)^{1/2}
.~~~
\end{eqnarray*}
Since the sequence $\{I_n\}$ is arbitrary, we conclude
\begin{eqnarray*}
\|T_g\|_{e, \mathcal{L}^{2,\lambda}\to B} &\gtrsim &\limsup_{|I|\rightarrow0} \big(\frac{1}{|I|^{p+1-\lambda}} \int_{S(I)} |g'(z)|^2(1-|z|^2)^{p}dm(z)\big)^{1/2}\\
&\thickapprox& \limsup_{|a|\rightarrow 1} \big( \int_{\D}(\frac{1-|a|^2}{|1-\overline{a}z|^2})^{p+1-\lambda} |g'(z)|^2(1-|z|^2)^{p}dm(z)\big)^{1/2}.
\end{eqnarray*}
This completes the proof.
\end{proof}

We have the following corollary about their compactness.
\begin{corollary}
Suppose that $0<\lambda<1$ and $p>1$. Then

$(i)$~$I_g:\mathcal{L}^{2,\lambda} \rightarrow B$ is compact if and only if $g=0$.

$(ii)$~$T_g:\mathcal{L}^{2,\lambda} \rightarrow B$ is compact if and only if $g\in B\mathcal{L}^{p,\lambda}_0$.
\end{corollary}

\section{Boundedness and essential norm of $I_g$ and $T_g$ from $\mathcal{L}^{2,\lambda}_0$ to $B_0$ }

\begin{theorem}\label{dl5}
Let $0<\lambda<1$ and $g\in H(\D)$. Then $I_g:\mathcal{L}^{2,\lambda}_0 \rightarrow B_0$ is bounded if and only if
$g\in H_{ \frac{\lambda-1}{2} }^\infty.$ Moreover, $$
\|I_g\| \thickapprox \|g\|_\frac{\lambda-1}{2}.$$
\end{theorem}

\begin{proof}Necessity: assume that $I_g:\mathcal{L}^{2,\lambda}_0 \rightarrow B_0$ is bounded. Then it is clear that $I_g:\mathcal{L}^{2,\lambda}_0 \rightarrow B$ is bounded. The necessity of Theorem
\ref{3dl1}, together with $f_b (z)\in \mathcal{L}^{2,\lambda}_0(\D)$, proves $g\in H_{ \frac{\lambda-1}{2} }^\infty.$

Sufficiency: let $g\in H_{ \frac{\lambda-1}{2} }^\infty$. Then from Theorem \ref{3dl1}, $I_g:\mathcal{L}^{2,\lambda} \rightarrow B$ is bounded and hence $I_g:\mathcal{L}^{2,\lambda}_0 \rightarrow B$ is bounded. It suffices to prove that for any $f\in \mathcal{L}^{2,\lambda}_0$, $I_g f\in B_0$. In fact, for any $f\in \mathcal{L}^{2,\lambda}_0$, we have
\begin{eqnarray*}
&&\lim_{|a|\to 1} \int_{\D} |f'(z)|^2|g(z)|^2(1-|\varphi_a (z)|^2)^{2-\lambda} dm(z)\\
&=&\lim_{|a|\to 1} (1-|a|^2)^{1-\lambda} \int_{\D} |f'(z)|^2(1-|\varphi_a (z)|^2)|g(z)|^2 (\frac{1-|z|^2}{|1-\overline{a}z|^2})^{1-\lambda}dm(z) \\
&\leq & \|g\|_\frac{\lambda-1}{2} \cdot  \lim_{|a|\to 1}(1-|a|^2)^{1-\lambda}\int_{\D} |f'(z)|^2(1-|\varphi_a (z)|^2)dm(z) =0.
\end{eqnarray*}
Consequently, $I_g:\mathcal{L}^{2,\lambda}_0 \rightarrow B_0$ is bounded.
\end{proof}

\begin{theorem}
Suppose that $0<\lambda<1$ and $g\in H(\D)$. If $I_g:\mathcal{L}^{2,\lambda}_0 \rightarrow B_0$ is bounded, then$$
\|I_g\|_{e, \mathcal{L}^{2,\lambda}_0 \to B_0} \thickapprox \|g\|_\frac{\lambda-1}{2}.$$
\end{theorem}
\begin{proof} As a matter of fact, if $g\in H_{ \frac{\lambda-1}{2} }^\infty$, then for any $f\in \mathcal{L}^{2,\lambda}_0$, $I_g f\in B_0$, Since  $f_n(z) \in \mathcal{L}^{2,\lambda}_0$ (see Theorem \ref{dl3}), we complete the proof  as the same as in the proof of Theorem \ref{dl3} .
\end{proof}

\begin{theorem}
\label{dl7}
Suppose that $0<\lambda<1$ and $g\in H(\D)$. Then the following conditions are
equivalent.

$(i)$ $T_g:\mathcal{L}^{2,\lambda}_0 \rightarrow B_0$ is bounded;

$(ii)$ $g\in B\mathcal{L}^{p,\lambda}_0$ for all $p\in (1,\infty)$;

$(iii)$ $g\in B\mathcal{L}^{p,\lambda}_0$ for some $p\in (1,\infty)$.

Moreover, $$
\|T_g\| \thickapprox M(g).$$
\end{theorem}

\begin{proof}

$(i)\Rightarrow (ii)$.  Suppose that $T_g$ is bounded. For any $I \subset \partial\mathbbm{D}$, let $b=(1-|I|)\zeta$, where $\zeta$ is the centre of $I$. Then
$$(1-|b|^2)\approx |1-\overline{b}z| \approx |I|, ~~z\in S(I).$$
Concerning the functions $F_b (z)$ in Lemma \ref{314}, we have $F_b (z)\in \mathcal{L}^{2,\lambda}_0(\D)$, thus $T_g(F_b (z))\in B_0$, Furthermore, from Proposition \ref{mt2} it yields $g\in B\mathcal{L}^{p,\lambda}_0$ .

$(ii)\Rightarrow (iii)$. It is obvious.

$(iii)\Rightarrow (i)$. Let fix $p\in (1,\infty)$ and $g\in B\mathcal{L}^{p,\lambda}_0$. Then from Theorem $\ref{dl2}$, $T_g:\mathcal{L}^{2,\lambda} \rightarrow B$ is bounded and hence $T_g:\mathcal{L}^{2,\lambda}_0 \rightarrow B$ is bounded. It suffices to prove that for any $f\in \mathcal{L}^{2,\lambda}_0$, $T_g f\in B_0$. Indeed, $g\in B\mathcal{L}^{p,\lambda}_0$, for every $\varepsilon >0$ there is an integer $\delta>0$ such that as $|J|<\delta$, $$\frac{1}{|J|^{p-\lambda+1}} \int_{S(J)} |g'(z)|^2(1-|z|^2)^{p}dm(z)<\varepsilon.$$
With the above $\delta$, for any $|I|$ satisfying $|I|<\delta$, we break up $S(I)$  in the same way in Theorem $\ref{dl2}$, then by $(\ref{32})$
\begin{eqnarray*}
\int_{S(I)} |g'(z)|^2(1-|z|^2)^{p+\lambda-1}dm(z)
&\lesssim &\sum_{n=0}^\infty\sum_{J\in \mathcal{D}_n(I)}\int_{S(J)} |J|^{\lambda-1}|g'(z)|^2(1-|z|^2)^{p}dm(z)\\
&\leq& \sum_{n=0}^\infty\sum_{J\in \mathcal{D}_n(I)}\varepsilon |J|^{\lambda-1}|J|^{p-\lambda+1}\\
&\lesssim& \varepsilon|I|^p.
\end{eqnarray*}
namely, $$\lim_{|I| \rightarrow 0} \frac{1}{|I|^p} \int_{S(I)} |g'(z)|^2(1-|z|^2)^{p+\lambda-1}dm(z)=0.$$
Now, it is easy to see that
\begin{eqnarray*}
\lefteqn{   \lim_{|I| \rightarrow 0} \big(\frac{1}{|I|^p} \int_{S(I)}|f(z)|^2 |g'(z)|^2(1-|z|^2)^{p}dm(z)\big)^{1/2}   } \hspace{30pt}\\
&\lesssim& \|f\|_{\mathcal{L}^{2,\lambda}}  \cdot \lim_{|I| \rightarrow 0} \big(\frac{1}{|I|^p} \int_{S(I)} |g'(z)|^2(1-|z|^2)^{p+\lambda-1}dm(z)\big)^{1/2}=0.
\end{eqnarray*}

In conclusion, $T_g:\mathcal{L}^{2,\lambda}_0 \rightarrow B_0$ is bounded.
\end{proof}

\begin{lemma}
\label{yl7}
 Suppose that $0<\lambda<1$, $1<p$ and $g\in B\mathcal{L}^{p,\lambda}_0$, the operator $T_{g,r}:\mathcal{L}^{2,\lambda}_0 \rightarrow B_0$ satisfies $$
T_{g,r} f(z)=\int_0^z f(rw)g'(w)dw.
$$ where $r\in (0,1)$. Then $T_{g,r}:\mathcal{L}^{2,\lambda}_0 \rightarrow B_0$ is compact.
\end{lemma}
\begin{proof}Since $g\in B\mathcal{L}^{p,\lambda}_0$, it follows from Lemma $\ref{yl6}$ that $T_{g,r}:\mathcal{L}^{2,\lambda} \rightarrow B$ is compact and hence $T_{g,r}:\mathcal{L}^{2,\lambda}_0 \rightarrow B$ is compact. As a matter of fact in Theorem $\ref{dl7}$, if $g\in B\mathcal{L}^{p,\lambda}_0$
, then for any $f\in \mathcal{L}^{2,\lambda}_0$, $T_g f\in B_0$, together with Lemma $\ref{yl5}$ , we conclude $f\in \mathcal{L}^{2,\lambda}_0$, $T_{g,r} f\in B_0$, so that $T_{g,r}:\mathcal{L}^{2,\lambda}_0 \rightarrow B_0$ is compact.
\end{proof}

\begin{theorem}
Let $0<\lambda<1$ and $g\in H(\D)$. If $T_g:\mathcal{L}^{2,\lambda}_0 \rightarrow B_0$ is bounded, then $$
\|T_g\|_{e, \mathcal{L}^{2,\lambda}_0\to B_0} \thickapprox \limsup_{|a|\rightarrow 1} \big( \int_{\D}(\frac{1-|a|^2}{|1-\overline{a}z|^2})^{p+1-\lambda} |g'(z)|^2(1-|z|^2)^{p}dm(z)\big)^{1/2}.$$
\end{theorem}
\begin{proof}
 Based on the fact that if $g\in B\mathcal{L}^{p,\lambda}_0$, then for any $f\in \mathcal{L}^{2,\lambda}_0$, $T_g f\in B_0$, together with $F_n(z) \in \mathcal{L}^{2,\lambda}_0$ and the compact operator $T_{g,r}$ in Lemma $\ref{yl7}$, similar to the proof of Theorem \ref{dl4}, we obtain the desired result.

\end{proof}
The following corollary is an immediate consequence of the above theorem.
\begin{corollary}
Let $0<\lambda<1$ and $p>1$. Then

$(i)$$I_g:\mathcal{L}^{2,\lambda}_0 \rightarrow B_0$ is compact if and only if $g=0$.

$(ii)$$T_g:\mathcal{L}^{2,\lambda}_0 \rightarrow B_0$ is compact if and only if $T_g:\mathcal{L}^{2,\lambda}_0 \rightarrow B_0$ is bounded if and only if $g\in B\mathcal{L}^{p,\lambda}_0$.
\end{corollary}


\begin{thebibliography}{123}
\bibitem{aaags}
A. Aleman and A.G. Siskakis, Integration operators on Bergman spaces, Indiana Univ. Math. J. 46 (1997) 337-356.
\bibitem{aajac}
A. Aleman and J.A. Cima, An integral operator on $H^p$ and Hardy's inequality, J. Anal. Math.85 (2001) 157-176.
\bibitem{npaom}
D. Adams and J. Xiao, Nonlinear potential analysis on Morrey spaces and their capacities,
Indiana Univ. Math. J. 53 (2004) 1629-1663.

\bibitem{miha}
D. Adams and J. Xiao, Morrey spaces in harmonic analysis, Ark. Mat.50 (2012) 201-230.


\bibitem{bf}
J.M. Anderson, Bloch functions: The basic theory, Operators and Funciton Theory, Math. Phys. Sci.153 (1985) 1-17.
\bibitem{obf}
J.M. Anderson, J. Clunie and Ch. Pommerenke, On Bloch functions and normal functions, J. r eine Angew. Math.270 (1974) 12-37.

\bibitem{radsjx}
R. Aulaskari, D.A. Stegenga and J. Xiao, Some subclasses of BMOA and their characterizations in terms of Carleson measures, Rocky. Mt. J. Math. 26 (1996) 485-506.
\bibitem{arlp}
R. Aulaskari and P. Lappan, Criteria for an analytic function to be Bloch and a harmonic or meromorphic function to be normal, In: Complex Analysis and Its Applications, Pitman Research Notes in Mathematics,305 (1994) 136-146.



\bibitem{bbt}
K.D. Bierstedt and W.H.Summers, Biduals of weighted Banach spaces of analytic functions,  J. Austral. Math. Soc.(Series A), 54 (1993) 70-79.
\bibitem{kabjbag}
K.D. Bierstedt, J. Bonet and A. Galbis, Weighted spaces of holomorphic functions on bounded domains, Michigan Math. J. 40 (1993) 271-297.


\bibitem{coosoaf}
C. Cowen and D. MacCluer, Composition operators on spaces of analytic functions, Studies in Advanced Mathematics. CRC Press, Boca Raton, FL, (1995).
\bibitem{tctiwh}
C. Cascante, J. F\`{a}brega and J. M. Ortega, The Corona theorem in weighted Hardy and Morrey spaces, Ann. Sc. Norm. Super. Pisa Cl. Sci. (5) DOI 10.2422/2036-2145.201202 006.
\bibitem{oc}
O. Constantin, A Volterra-type integration operator on Fock spaces, Proc. Amer. Math. Soc. 140 (2012) 4247-4257.



\bibitem{dur}
P. Duren, Theory of $H^p$ Spaces, Academic Press, New York, (1970).
\bibitem{gir}
D. Girela, Analytic functions of bounded mean oscillation, Complex Function Spaces (Mekrijarvi, 1999), 61-170, Univ. Joensuu Dept. Rep. Ser., 4, Univ. Joensuu, Joensuu, 2001.

\bibitem{jlsmpn}
J. Laitila, S. Miihkinen and P. Nieminen, Essential norms and weak compactness of integration operators, Arch. Math. 97 (2011) 39-48.
\bibitem{jlzlcx}
J. Liu, Z. Lou and C. Xiong, Essential norms of integral operators on spaces of analytic functions, Nonlinear Anal. 75 (2012)  5145-5156.

\bibitem{pljlzl}
P.T. Li, J.M. Liu and Z.J. Lou, Integral operators on analytic Morrey spaces, Sci China Math. http://arxiv.org/abs/1304.2575


\bibitem{slss}
S. Li and S. Stevi\'{c}, Volterra-type operators on Zygmund spaces, J. Ineq. Appl.  Article ID 32124, (2007) 10 pages

\bibitem{cbmj}
C.B. Morrey Jr., On the solutions of quasi-linear elliptic partial differential equations, Trans. Amer. Math. Soc. 43 (1938), 126-166.
\bibitem{km}
K. Madigan and A. Matheson, Compact composition operators on the Bloch space. Tran. Amer. Math. Soc. 347 (1995) 2679-2687.




\bibitem{chp}
Ch. Pommerenke, Schlichte Funktionen und analytische Funktionen von beschr$\ddot{a}$nkter mittlerer Oszillation, Comment.Math. Helvetici. 52 (1977) 591-602.

\bibitem{jp}
J. Peetre, On the theory of  $\mathcal{L}_{p,\lambda} $ spaces, J. Funct. Anal. 4 (1964) 71-87.




\bibitem{agsrz}
A. Siskakis and R. Zhao, A Volterra type operator on spaces of analytic functions, Contemporary Mathematics. 232 (1999) 299-311.
\bibitem{sw1}
A.L. Shields and D.L. Williams, Bounded projections, duality, and multipliers in spaces of harmonic functions, J. Reine Angew. Math. 299 (1978)  256-279.
\bibitem{sw2}
A.L. Shields and D.L. Williams, Bounded projections and the growth of harmonic conjugates in the disk, Michigan Math. J.29 (1982) 3-25.
\bibitem{js}
J. Shapiro, The essential norm of a composition operator, Ann. Math. 125 (1987) 375-404.

\bibitem{zw}
Z. Wu,  A new characterization for Carleson measure and some applications, Integr. Equ. Oper. Theory, 71 (2011) 161-180.

\bibitem{hwjz}
Z. Wu and   C. Xie, Q spaces and Morrey spaces, J. Funct. Anal, 297 (2003) 282-297.



\bibitem{jxqcm}
J. Xiao, The $Q_p$ Carleson measure problem,  Adv. Math. 217 (2008) 2075-2088.

\bibitem{jx}
J. Xiao, Holomorphic Q classes. Lecture Notes in Mathematics, Springer-Verlag, Berlin. 1767 (2001).
\bibitem{xia1}
J. Xiao, Geometric Qp Functions, Frontiers in Mathematics. Birkh$\ddot{a}$auser Verlag, Basel, 2006.
\bibitem{cobacs}
J. Xiao and W. Xu, Composition operators between analytic Campanato spaces, J. Geom. Anal, Doi 10.1007/s12220-012-9349-6.
\bibitem{wul}
H. Wulan  and J. Zhou, QK and Morrey type spaces. Ann Acad Sci Fenn Math. 38 (2013) 193-207. 
\bibitem{povo}
S. Ye, Products of Volterra-type operators and composition operators on logarithmic Bloch space, WSEAS Trans. Math. 12 (2013) 180-188.
\bibitem{ysgj}
S. Ye and J. Gao, Extended Ces\'{a}ro Operators Between Diferent Weighted Bloch―type Spaces, Acta Math. Sci.
(Series A), 28 (2008) 349-358.(In Chinese).

\bibitem{ctz}
C.T. Zorko, Morrey space, Proc. Amer. Math. Soc. 98 (1986) 586-592.
\bibitem{otifs}
K. Zhu, Operator theorey in function spaces, Senond edition, Mathematical surveys and Monographs 138 (2007).

\end{thebibliography}
\end{document}